\pdfoutput=1

\documentclass[showpacs,nofootinbib,showkeys,eqsecnum,prd,aps,preprint]{revtex4-2}

\usepackage[english]{babel}
\usepackage{comment}

\usepackage{amssymb}
\usepackage{amsmath}
\usepackage{amsfonts}
\usepackage{latexsym}
\usepackage{mathrsfs}
\usepackage{bm}
\usepackage{tensor}
\usepackage{hyperref}
\usepackage{graphicx}

\def\const{\mathop{\rm const}\nolimits\,}
\def\dac{\displaystyle\frac}

\usepackage{amsthm}
\newtheorem{theorem}{Theorem}[section]
\newtheorem{lemma}[theorem]{Lemma}
\newtheorem{corollary}[theorem]{Corollary}

\usepackage{diagbox}

\def\const{\mathop{\rm const}\nolimits\,}

\def\dac{\displaystyle\frac}

\def\{{\lbrace}
\def\}{\rbrace}

\begin{document}

\title{Analytical description of the diffusion in a cellular automaton with the Margolus neighbourhood in terms of the two-dimensional Markov chain}

\author{Anton E. Kulagin}
\email{aek8@tpu.ru}
\affiliation{Division for Electronic Engineering, Tomsk Polytechnic University, 30 Lenina av., 634050 Tomsk, Russia}

\author{Alexander V. Shapovalov}
\email{shpv@phys.tsu.ru}
\affiliation{Department of Theoretical Physics, Tomsk State University, Novosobornaya Sq. 1, 634050 Tomsk, Russia}
\affiliation{Laboratory for Theoretical Cosmology, International Centre of Gravity and Cosmos, Tomsk State University of Control Systems and Radioelectronics, 40 Lenina av., 634050 Tomsk, Russia}

\begin{abstract}
The one-parameter two-dimensional cellular automaton with the Margolus neighbourhood is analyzed based on the considering the projection of the stochastic movements of a single particle. Introducing the auxiliary random variable associated with the direction of the movement, we reduce the problem under consideration to the study of a two-dimensional Markov chain. The master equation for the probability distribution is derived and solved exactly using the probability generating function method. The probability distribution is expressed analytically in terms of Jacobi polynomials. The moments of the obtained solution allowed us to derive the exact analytical formula for the parametric dependence of the diffusion coefficient in the two-dimensional cellular automaton with the Margolus neighbourhood. Our analytic results agree with earlier empirical results of other authors and refine them. The results are of interest for the modelling two-dimensional diffusion using cellular automata especially for the multicomponent problem.\\
\end{abstract}



\maketitle

\section{Introduction}
\label{sec:int}

Cellular automata (CA) are the powerful tool for modelling the matter transport \cite{wolfram2002}. The attractive feature of CA is that contemporary algorithms of parallel computing can be naturally embedded in them. Various cellular algorithms were successfully implemented on field-programmable gate arrays (FPGA) \cite{Palchaudhuri2022,Cicuttin2021} and on graphic cards \cite{Cagigas2022,Matolygin2020}.

In the study of diffusion phenomena in physics, chemistry, and biology, CA are positioned as one of the effective tools for modelling systems of many particles. The simplest CA with boolean alphabet modelling the diffusion are the asynchronous CA with naive diffusion and the synchronous CA with Margolus neighbourhood \cite{margolus87} or briefly Margolus CA (MCA). The last one is the object of our study. Exactly this type of the CA attracted our attention since synchronous CA are more advantageous in terms of the computing speed (see, e.g., comparison in \cite{Kireeva2019345}). Note that the MCA is the representative of partitioning CA.

The most common macroscopic characteristics of the diffusion transfer is the diffusion coefficient ${\mathcal{D}}_c$ that is defined as the proportionality factor in Fick`s law \cite{fick95,paul14}. The diffusion coefficient is usually invariant when the movements of particles are caused just by random fluctuations rather than by any fields. Nevertheless, the processes with the time dependent diffusion coefficient are also of interest, e.g., in the ambipolar diffusion phenomena \cite{shapkul21,shapkul22,editorial22}. The most popular model of the diffusion is the kinetic model, which is based on a differential equations that can be solved, for example, using difference schemes \cite{mickens2000,pankov21,shokri16}. In such approach, the diffusion coefficient is an explicit parameter of the model. For cellular automata, the diffusion coefficient is not presented in the model explicitly. Calculating the diffusion coefficient for the CA with naive diffusion is a trivial task since it directly models the random walk of particles that is well-studied. However, it is not trivial for partitioning CA including MCA. We show here that the diffusion in the two-dimensional MCA can be described in terms of one two-dimensional Markov chain as opposed to the one-dimensional random walk. In our paper, we study the exact solution of this two-dimensional Markov chain whose characteristics are crucial for the mathematical description of the diffusion in MCA.

While multidimensional Markov chains arise in a number of applications in service theory \cite{Rachinskaya2018}, genetic networks \cite{Ching2004}, Gibbs sampling \cite{morzfeld19}, and other areas \cite{Ching2008}, there are very few exactly solvable multidimensional Markov chains that do not degenerate to the one-dimensional one. Therefore, the exact solutions to the two-dimensional Markov chain under consideration can be of interest for specialists in applied mathematics and useful for applications other than cellular automaton theory.

The paper is organized as follows. In Section \ref{sec:mca}, we describe the diffusion model within the framework of the two-dimensional cellular automata with Margolus neighbourhood. Also, we explain the way of obtaining the diffusion coefficient for this model based on the considering one-dimensional movement of a single particle. In Section \ref{sec:single}, the stochastic process describing the one-dimensional movement of a single particle is given. The idea of introducing the auxiliary random variable associated with the direction of movement is proposed. This approach allows us to consider the process under consideration as the two-dimensional Markov chain. In Section \ref{sec:math1}, the formal mathematical definition of this chain and basic notations are given. In Section \ref{sec:math2}, we solve the master equation for the probability distribution of the particle coordinate using the probability generating function method. The moments of this distribution are derived. In Section \ref{sec:math3}, the obtained results are discussed within the framework of the diffusion model. The exact analytical formula for the diffusion coefficient is derived, and properties of the cellular automaton with the Margolus neighbourhood are discussed. In Section \ref{sec:concl}, we conclude with some remarks.

\section{Diffusion in the Margolus cellular automaton}
\label{sec:mca}

In the two-dimensional cellular automata with Margolus neighbourhood, the diffusion is described based on the following algorithm \cite{margolus87}:\\
1) The plane is divided into identical square cells forming a cell grid. One cell contains a boolean variable that takes the value 0 or 1 corresponding to the absence or presence of a {\it particle} in the cell respectively.\\
2) The cell grid is divided into blocks of $2\times 2$ cells. Such partition can be even or odd (see Fig. \ref{fig1}). The even partition corresponds to the case when the left lower cell of each block has even vertical and horizontal numbers. The odd partition corresponds to the case when the left lower cell of each block has odd vertical and horizontal numbers.\\
3) Beginning from an initial state, on each discrete time step the following {\it rule} is applied to each block independently. We rotate the block $90^\circ$ clockwise or counterclockwise with the probability of $
0<p\leq 0.5$ for each outcome. The {\it rule} is applied to blocks of the even partition on even time steps and it is applied to blocks of the odd partition on odd time steps. If each time step corresponds to the same time interval $\Delta t$ and $p=\const$, then such MCA simulates the diffusion with the time independent diffusion coefficient ${\mathcal{D}}_c$.


The classic MCA uses $p=0.5$, i.e. blocks always rotate. In such approach, the diffusion rate is controlled with the length of the time interval corresponding to one discrete time step and with the size of a cell. However, this model can be applied to the multicomponent diffusion problem. In this case, the cell contains few bits that correspond to {\it particles} of different substances. It can be treated as few layers of the cell grid. Then, we can not directly apply the described {\it rule} with $p=0.5$ for every layer of {\it particles} if the diffusion rates for these substances are not equal. The difference in diffusion rates can be realized using different $p$ for different substances. The only drawback of such approach is that the diffusion coefficient ${\mathcal{D}}_c={\mathcal{D}}_c(p)$ nonlinearly depends on $p$. The other way is using the same $p=0.5$ for each layer but skip the two consecutive time steps for one or few layers with a probability equal to $p_s$ (the rotation {\it rule} is not applied to any blocks of the respective layer on the skipped time step). In this approach, the diffusion coefficient linearly depends on $p_s$. The differences between these two approaches will be discussed in Section \ref{sec:math3}.


In order to model the diffusion using the MCA, the diffusion coefficient corresponding to the particular MCA must be determined. It can be shown (see, e.g., \cite{bandman1999}) that
\begin{equation}
{\mathcal{D}}_c(p)=k(p)\dac{\Delta x^2}{\Delta t},
\label{mca1}
\end{equation}
where ${\mathcal{D}}_c(p)$ is the diffusion coefficient, $\Delta x$ is the side length of the cell, $\Delta t$ is the time interval corresponding to the discrete time step, and $k(p)$ is the proportionality factor. However, relation \eqref{mca1} does not help to obtain the exact value of the multiplier $k$. Moreover, $k$ nonlinearly depends on $p$ as it was mentioned earlier. One way is to determine $k$ is to compute the specific diffusion problem using the numerical realization of the MCA and to compare the result with the analytical or numerical solutions of the respective diffusion equation. Such approach allows one to find $k$ approximately. However, it is hard to estimate error of obtained value. Moreover, this error depends on the choice of the specific diffusion problem. The other way is to obtain the distribution of the {\it particle} position in the MCA. Assuming that the correlation of these distributions for two {\it particles} tends to zero with the growth of time steps number, the distribution of the position of the single {\it particle} is equal to the Green function of the diffusion equation that is given by
\begin{equation}
G(\vec{r},\tau)\sim\exp\left(-\dac{\vec{r}\,^2}{4{\mathcal{D}}_c \tau}\right),
\label{mca2}
\end{equation}
where $\vec{r}$ is the radius-vector of the particle position at time $t=\tau$ relative to the initial position at time $t_0=0$, ${\mathcal{D}}_c$ is the diffusion coefficient. For the isotropic and homogeneous problem, one can consider the projection of the {\it particle} position on one axis (let it be $x$).
As $t\to \infty$, the distribution of the {\it particle} position tends to the normal distribution with the probability density function given by
\begin{equation}
P_{X_n}(x)\sim\exp\left(-\dac{(x-M_{X_n})^2}{2D_{X_n}}\right), \qquad t=n\Delta t, \qquad n\in {\mathbb{N}},
\label{mca3}
\end{equation}
where $X_n$ is the random position of the {\it particle} at the $n$-th time step, $M_{X_n}$ is the expectation of $X_n$, and $D_{X_n}$ is the dispersion of $X_n$. Then, for $\Delta x=1$, $\Delta t=1$, we have
\begin{equation}
\begin{gathered}
M_{X_n}\to 0, \qquad n\to \infty,\\
\dac{D_{X_n}}{2n}\to {\mathcal{D}}_c, \qquad n\to \infty.
\end{gathered}
\label{mca4}
\end{equation}
Hereinafter, ${\mathcal{D}}_c$ will be given for $\Delta x=1$, $\Delta t=1$ unless otherwise stated. The trick is to find the distribution $P_{X_n}(x)$ or its moments analytically. In \cite{malin98}, the authors tried to obtain the diffusion coefficient using this approach for $p=0.5$. However, they assumed that $X_n$ is the one-dimensional Markov chain, i.e. the simplest random walk. Later on, using the software realization of MCA with higher number of cells and time steps, it was shown that this assumption led to that the diffusion coefficient obtained in \cite{malin98} for 2D MCA is approximately three times higher than the real one (see, e.g., \cite{bandman1999,Bandman2012,Shalyapina2021}). However, the exact distribution of $X_n$ and the diffusion coefficient of MCA were not obtained since then. In this work, we present the exact distribution of $X_n$ in MCA and obtain the exact analytical formula of ${\mathcal{D}}_c(p)$ for an arbitrary $p$ including the special case $p=0.5$.

\section{One-dimensional movement of a single particle}
\label{sec:single}

In this Section, we will describe the rules of the stochastic movement of a single {\it particle} in MCA. As it was mentioned in the previous Section, we consider the movement along the $x$-axis that is directed along the side of a cell. Notice that if a {\it particle} is in an odd column at an odd time step, then its $x$-coordinate can either increase by 1 or do not change according to the rotation {\it rule} of the MCA (see Fig.~\ref{fig2}). The same rule works for a {\it particle} in a even column at an even time step. Let the variable $d$ show the direction of possible movement of the {\it particle} along the $x$-axis. If $d=1$, then $x$ can either increase by 1 or do not change. If $d=-1$, then $x$ can either decrease by 1 or do not change. It is easy to see that $d=-1$ corresponds to a {\it particle} in an even column at an odd time step and to a {\it particle} in an odd column at an even time step.


Notice that the direction $d$ does not change when the $x$-coordinate changes and the direction $d$ reverts its sign when the coordinate $x$ does not change. The first outcome is realized with the probability of $p$ at every time step while the second outcome is realized with the probability $(1-p)$. Note that the sequence of the $x$-coordinates of the single {\it particle} has a memory, i.e. the current $x$-coordinate of the {\it particle} depends the history of its movements. It the main issue in the analytical description of partitioning cellular automata.

\section{Master equation}
\label{sec:math1}

Next, we give the formal definition of the two-dimensional Markov chain described in Section \ref{sec:single}.

Let $X_t:\Omega\to {\mathbb{Z}}$ and $\Delta_t:\Omega\to \{-1,+1\}$ be sequences of random variables where $t\in{\mathbb{Z}}_+$ and $\Omega$ is a set of outcomes. We associate the values of $X_t$ with the discrete $x$-coordinate of the {\it particle} on a discrete time step $t$, and $\Delta_t$ is associated with its direction. The probability of transitions for the chain under consideration, which were described in Section \ref{sec:single}, are given by
\begin{equation}
\begin{gathered}
{\rm Prob}(X_{t+1}=X_t+\Delta_t; \Delta_{t+1}=\Delta_t)=p, \\
{\rm Prob}(X_{t+1}=X_t; \Delta_{t+1}=-\Delta_t)=1-p.
\end{gathered}
\label{def1}
\end{equation}
The semicolon in \eqref{def1} implies the logical conjunction.
Let us denote
\begin{equation}
P_t(x,d)={\rm Prob}(X_{t}=x;\Delta_t=d).
\label{def3}
\end{equation}
The transition rules \eqref{def1} in notations \eqref{def3} yield the following master equation:
\begin{equation}
P_{t+1}(x,d)=p P_t(x-d,d)+(1-p)P_t(x,-d).
\label{def4}
\end{equation}
Hereinafter, $x\in{\mathbb{Z}}$, $d\in\{-1,+1\}$, and $t\in{\mathbb{Z}}_+$. We consider the problem with the deterministic initial position of the {\it particle}, i.e. ${\rm Prob}(X_0=0)=1$. Then, the initial condition for \eqref{def4} reads
\begin{equation}
P_0(x,+1)=\varepsilon\cdot \delta_{x0}, \qquad P_0(0,-1)=(1-\varepsilon)\cdot \delta_{x0},
\label{def5}
\end{equation}
where $\delta_{ij}$ is the Kronecker delta.

Within the framework of the diffusion problem, it is natural to consider the symmetric case $P_0(0,+1)=P_0(0,-1)$. Thus, the following initial condition will be considered:
\begin{equation}
P_0(x,+1)=\dac{\delta_{x0}}{2}, \qquad P_0(x,-1)=\dac{\delta_{x0}}{2}.
\label{def6}
\end{equation}
Since the distribution over $x$ is of primary interest to us, we will study the marginal distribution given by
\begin{equation}
P_t(x)=P_t(x,+1)+P_t(x,-1).
\label{def7}
\end{equation}
In order to construct the analytical description of the two-dimensional Markov chain under consideration, we will solve the master equation \eqref{def4} with the initial condition \eqref{def6} in the next Section.

\section{Probability generating function}
\label{sec:math2}
The exact solution of \eqref{def4}, \eqref{def6} can be derived using the method of the probability generating function \cite{nelson1995} for the marginal distribution \eqref{def7}. Such probability generating function reads
\begin{equation}
G_t(z)=\sum_{x=-\infty}^{\infty}P_t(x)z^x.
\label{prgen1}
\end{equation}
Let us also introduce the following probability generating functions:
\begin{equation}
G_t^{+}(z)=\sum_{x=-\infty}^{\infty}P_t(x,+1)z^x, \qquad G_t^{-}(z)=\sum_{x=-\infty}^{\infty}P_t(x,-1)z^x.
\label{prgen2}
\end{equation}
In view of \eqref{def7}, one readily gets
\begin{equation}
G_t(z)=G_t^{+}(z)+G_t^{-}(z).
\label{prgen3}
\end{equation}
The master equation \eqref{def4} can be written as
\begin{equation}
\left\{\begin{array}{l}
P_{t+1}(x,+1)=p P_t(x-1,+1)+(1-p)P_t(x,-1),\cr
P_{t+1}(x,-1)=p P_t(x+1,-1)+(1-p)P_t(x,+1).
\end{array}\right.
\label{prgen4}
\end{equation}
Multiplying the system \eqref{prgen4} by $z^x$ and summing over $x\in{\mathbb{Z}}$, we derive the following equation for the probability generating functions \eqref{prgen2}:
\begin{equation}
\begin{gathered}
g_{t+1}(z)=b(z)g_t(z), \qquad g_0(z)=\dac{1}{2}\begin{pmatrix} 1 \\ 1 \end{pmatrix},\\
g_t(z)=\begin{pmatrix} G^{+}_t(z) \\ G^{-}_t(z) \end{pmatrix}, \qquad b(z)=\begin{pmatrix} pz & 1-p \\ 1-p & pz^{-1} \end{pmatrix},
\end{gathered}
\label{prgen5}
\end{equation}
where $g_0(z)$ is obtained from \eqref{def6}, \eqref{prgen2}.

The solution of the system \eqref{prgen5} has the following form:
\begin{equation}
g_t(z)=u(z)\begin{pmatrix}\lambda_1^t(z) & 0 \\ 0 & \lambda_2^t(z) \end{pmatrix}u^{-1}(z)g_0(z), \qquad u(z)u^{\top}(z)=u^{\top}(z)u(z)=1,
\label{prgen6}
\end{equation}
where the solution of the associated matrix eigenvalue problem yields
\begin{equation}
\begin{gathered}
\lambda_{1}(z)=\dac{p}{2}\left(z+\dac{1}{z}\right)+ \dac{m(z)}{2}, \qquad  \lambda_{2}(z)=\dac{p}{2}\left(z+\dac{1}{z}\right)- \dac{m(z)}{2}, \\
u(z)=\begin{pmatrix} c_1(z)(1-p)& c_2(z)(1-p) \\ \dac{c_1(z)}{2}\left(m(z)-r(z)\right) & -\dac{c_2(z)}{2}\left(m(z)+r(z)\right) \end{pmatrix}, \qquad m(z)=\sqrt{p^2\left(z+\dac{1}{z}\right)^2-8p+4},\\
r(z)=p\left(z-\dac{1}{z}\right), \qquad c_1(z)=\sqrt{\dac{m(z)+r(z)}{2(1-p)^2m(z)}}, \qquad c_2(z)=\sqrt{\dac{m(z)-r(z)}{2(1-p)^2m(z)}}.
\end{gathered}
\label{prgen7}
\end{equation}

The solution \eqref{prgen6}, \eqref{prgen7} can be rewritten as
\begin{equation}
\begin{gathered}
G^{+}_t(z)=\dac{1}{2m(z)}\left[-\lambda_1^t(z)\lambda_2(z)+\lambda_1(z)\lambda_2^t(z)+(1-p+pz)\left( \lambda_1^t(z)-\lambda_2^t(z) \right)\right],\\
G^{-}_t(z)=\dac{1}{2m(z)}\left[\lambda_1^{t+1}(z)-\lambda_2^{t+1}(z)+(1-p-pz)\left( \lambda_1^t(z)-\lambda_2^t(z) \right)\right].
\end{gathered}
\label{prgen8}
\end{equation}
Raw moments of $X_t$ for the given $\Delta_t$ read
\begin{equation}
\mu^{(n)\pm}_t=\sum_{x\in{\mathbb{Z}}}x^n P_t(x,\pm 1).
\label{prgen9}
\end{equation}
The moments \eqref{prgen9} can be obtained from
\begin{equation}
\mu^{(n)\pm}_t=\dac{d^n G_t^{\pm}(z)}{dz^n}\Big|_{z=1}.
\label{prgen10}
\end{equation}
Using \eqref{prgen10}, \eqref{prgen8}, one readily gets
\begin{equation}
\begin{gathered}
\mu^{(1)+}_t=-\mu^{(1)-}_t=\dac{p}{4-4p}\left(1-(2p-1)^t\right),\\
\mu^{(2)+}_t=\mu^{(2)-}_t=\dac{p^2}{4(1-p)^2}\left(-1+\dac{2(1-p)}{p}t+(2p-1)^t\right).
\end{gathered}
\label{prgen11}
\end{equation}
Then, the expectation $M_{X_t}$ and the dispersion $D_{X_t}$ of the marginal distribution \eqref{def7}, which are given by
\begin{equation}
M_{X_t}=\sum_{x\in{\mathbb{Z}}}x P_t(x), \qquad D_{X_t}=\sum_{x\in{\mathbb{Z}}}(x-M_{X_t})^2 P_t(x),
\label{prgen12}
\end{equation}
read
\begin{equation}
M_{X_t}=0, \qquad D_{X_t}=\dac{p^2}{2(1-p)^2}\left(-1+\dac{2(1-p)}{p}t+(2p-1)^t\right).
\label{prgen13}
\end{equation}
While the formula \eqref{prgen8} is convenient for the derivation of moments, it does not allow one to directly obtain the distribution \eqref{def7}. The explicit representation of the probability generating function as the power series in $z$ is a bit tricky. The analytical form for the marginal distribution \eqref{def7} is given by the following theorem
\begin{theorem}
The marginal distribution $P_t(x)$ reads
\begin{align}
\label{teo1}
&P_{2n+1}(2j)=(1-p)p^{2j}(1-2p)^{n-j}P_{n-j}^{(2j,0)}\!\left[\dac{2p^2}{1-2p}+1\right], \qquad j=\overline{0,n};\\
\label{teo2}
&P_{2n}(2j)=(1-p)^{2} p^{2j}(1-2p)^{n-j-1}\dac{n}{n-j}P_{n-j-1}^{(2j,1)}\!\left[\dac{2p^2}{1-2p}+1\right], \qquad j=\overline{0,n-1};\\
\label{teo3}
&P_{2n}(2j+1)=(1-p)p^{2j+1}(1-2p)^{n-j-1}P_{n-j-1}^{(2j+1,0)}\!\left[\dac{2p^2}{1-2p}+1\right], \qquad j=\overline{0,n-1};\\
\label{teo4}
&P_{2n+1}(2j+1)=(1-p)^2p^{2j+1}(1-2p)^{n-j-1}\dac{n+\frac{1}{2}}{n-j}P_{n-j-1}^{(2j+1,1)}\!\left[\dac{2p^2}{1-2p}+1\right],\\
\nonumber
& j=\overline{0,n-1};\\
\label{teo5}
&P_{n}(n)=\dac{1}{2}p^n;\\
\nonumber
&P_{t}(-x)=P_t(x), \qquad n\in {\mathbb{N}}.
\end{align}
where $P_n^{\alpha,\beta}[x]$ are the Jacobi polynomials, and $C^k_n=\dac{n!}{k!(n-k)!}$ are the binomial coefficients.

For $p=\dac{1}{2}$, the formulae \eqref{teo1}, \eqref{teo2}, \eqref{teo3}, \eqref{teo4} yield
\begin{equation}
\begin{gathered}
P_{2n+1}(2j)=2^{-2n}C^{j+n}_{2n}, \qquad j=\overline{0,n};\\
P_{2n}(2j)=2^{-2n}C^{j+n}_{2n}, \qquad j=\overline{0,n-1}\\
P_{2n}(2j+1)=2^{-2n+1}C^{j+n}_{2n-1}, \qquad j=\overline{0,n-1}\\
P_{2n+1}(2j+1)=2^{-2n-1}C^{j+n+1}_{2n+1}, \qquad j=\overline{0,n-1}.
\end{gathered}
\label{teo6}
\end{equation}
\label{teor1}
\end{theorem}
The proof of this theorem is given in Appendix \ref{sec:app}.

\section{Discussion of the results}
\label{sec:math3}
Now, from \eqref{mca4} and \eqref{prgen13}, we can obtain the diffusion coefficient ${\mathcal{D}}_c(p)$:
\begin{equation}
{\mathcal{D}}_c(p)=\lim_{t\to \infty} \dac{D_{X_t}}{2t}=\dac{1}{2}\dac{p}{1-p}.
\label{disc1}
\end{equation}
In Fig. \ref{fig3}, \ref{fig4}, \ref{fig5}, the probability distribution from \eqref{teo1}, \eqref{teo2}, \eqref{teo3}, \eqref{teo4} (or \eqref{teo6} for $p=\frac{1}{2}$), and \eqref{teo5} is plotted along with the following normal probability distribution:
\begin{equation}
f_t(x)=\sqrt{\dac{1-p}{2t\pi p}}\exp\left(-\dac{(1-p)x^2}{2tp}\right).
\label{disc2}
\end{equation}
Note that in view of {\eqref{mca3}}, {\eqref{mca4}} the distribution $P_t(x)$ tends to $f_t(x)$ as $t\to \infty$ at least for $p\in[0;\frac{1}{2}]$. The moments of the normal distribution {\eqref{disc2}} are given by {\eqref{disc1}}, {\eqref{prgen13}}.

%
%

From Fig. \ref{fig3}, \ref{fig4}, it is clear that the distribution is very close the normal one even for small $t$.

Note that the probability $p>\dac{1}{2}$ does not make sense for MCA. However, the definition \eqref{def1} of the two-dimensional Markov chain under consideration admits $p\in(0;1)$ ($p=0$ and $p=1$ are trivial cases), and Theorem \ref{teor1} holds for $p\in(0;1)$. For completeness, we have shown the probability distribution for $p>\dac{1}{2}$ in Fig. 5. Apart from the fact that $p>\dac{1}{2}$ has no physical meaning regarding the diffusion, we can note that the probability distribution $P_t(x)$ is a nonmonotonic function with respect to $x\geq 0$ for $p>\dac{1}{2}$.

Next, we will discuss the differences between the two types of MCA mentioned in Section \ref{sec:mca}. Let the MCA parameterized by $p$ be termed the first type MCA. The second type MCA corresponds to the $p=\dac{1}{2}$ but it is parameterized by $p_s$. The parameter $p_s$ determines the probability that the rotation {\it rule} does not applies to any block of the odd and even partition at the next two time steps. In Section \ref{sec:mca}, we described it as the probability of skipping two consecutive time steps. We will limit ourselves to the simple case when the probability $p_s$ is checked only at odd time steps. Note that $p_s$ corresponds to the global rule while the probability $p$ corresponds to the local {\it rule} since it is checked for each individual block independently.

We have shown that the diffusion coefficient ${\mathcal{D}}_c(p)$ \eqref{disc1} in the first type MCA nonlinearly depends on $p$. On the other hand, the diffusion coefficient $\tilde{{\mathcal{D}}_c}(p_s)$ in the second type MCA linearly depends on $p_s$ and is given as follows:
\begin{equation}
\tilde{{\mathcal{D}}_c}(p_s)=(1-p_s){\mathcal{D}}_c(0.5).
\label{disc3}
\end{equation}
Both of these MCA can be adjusted to the same diffusion coefficient by the appropriate choice of $p$ and $p_s$, i.e. they have the same macroscopic behaviour at large times. However, they behave differently at small times $t$. Since the first and the third moments of the distribution given by \eqref{teo1}, \eqref{teo2}, \eqref{teo3}, \eqref{teo4}, \eqref{teo5} equal to zero, the core information about the process behaviour is carried by the dispersion. The behaviour of the first type MCA corresponds to the dispersion $D_{X_t}(p)$ given by \eqref{prgen13} while the behaviour of the second type MCA corresponds to the dispersion $\tilde{D}_{X_t}(p_s)$ given by
\begin{equation}
\tilde{D}_{X_t}(p_s)=(1-p_s) \cdot D_{X_t}(0.5).
\label{disc3a}
\end{equation}
In Fig. \ref{fig6}, the time dependence of $D_{X_t}(p)$ and $\tilde{D}_{X_t}(p_s)$ are shown.


We see that the dispersion $D_{X_t}(p)$ tends to its large time asymptotic value faster than $\tilde{D}_{X_t}(p_s)$. Moreover, for small $t$, it is also much closer to this value that determines the diffusion coefficient in MCA especially for $p<\dac{1}{2}$. It means that the first type MCA shows its macroscopic diffusion behaviour at smaller $t$. That can be treated as a better time resolution of such MCA. Note that the comparison of the dispersions yields just an lower estimate for the time resolution since we do not take into account correlation between the movement of different particles in MCA. However, there is a reason to believe that this correlation is even higher for the second type MCA with $p_s\neq 1$ since the skipping of the two consecutive time steps is performed for the whole cellular grid as opposite to the first type MCA, where it is checked for the every block whether it rotates or not on the current time step. Hence, the difference mentioned above may be even greater than one obtained from our estimates.

Let us compare our analytical results with some empirical results of other authors.

In \cite{bandman1999}, the author has compared the results of modelling with MCA and solutions of PDE numerically. The diffusion coefficient ${\mathcal{D}}_c(0.5)=0.512$ was obtained. The error of 5\% was declared in \cite{bandman1999}. Their numerical results agrees with the exact value of $0.5$ obtained in this work.


In \cite{Shalyapina2021}, the authors have empirically obtained the dependence of ${\mathcal{D}}_c(p)$ (in relative units) on $p$. They constructed the following regression model for this dependence:
\begin{equation}
p=-0.35\left(\dac{{\mathcal{D}}_c(p)}{{\mathcal{D}}_c(0.5)}\right)^2+0.86\dac{{\mathcal{D}}_c(p)}{{\mathcal{D}}_c(0.5)}.
\label{dep1}
\end{equation}
Their experimental values for \eqref{dep1} fit well to the curve \eqref{disc1} taking into account the inaccuracy of their method of matching the solutions of MCA and PDE. Hence, they could obtain the regression model that is close to the exact formula if they would use the rational function approximation, which is widely used in many applications \cite{gluzman2020}, instead of the polynomial model.

Note that along with the MCA with the Boolean alphabet the one with the integer alphabet is also used \cite{Medvedev2010,Kireeva2019345}. Such approach allows one to lower the concentration noise \cite{Bandman2005} when modelling reaction-diffusion systems. In order to control the diffusion coefficient independently of reaction rates in such MCA, the rotation {\it rule} applies not to the whole integer value but to its percentage. Applying the rotation {\it rule} to the percentage of {\it particles} is equivalent to the changing the block rotation probability $p$ by the same factor in terms of the diffusion coefficient. Thus, the dependence of the diffusion coefficient on the percentage $\xi$ of {\it particles} that are involved in the rotation has the same form as the dependence \eqref{disc1}  up to changing $\xi=2p$.

We should note that, in some works related to the multicomponent MCA, reseachers falsely suppose that the dependence ${\mathcal{D}}_c(p)$ should be linear for a wide range of $p$. This assumption originates from the work \cite{chopard94}, where the authors claimed that ${\mathcal{D}}_c(p)$ is linear on $p$ even when the rotation probability is independent for each block. Actually, they considered the modified version of the MCA that is neither the MCA with boolean alphabet, nor the multicomponent MCA. Our exact analytical formula {\eqref{disc1}} shows that the linear approximation is inaccurate when $p$ is substantially lower than $\frac{1}{2}$. Note that it also agrees with the cited numerical results \cite{Shalyapina2021}.


\section{Conclusion}
\label{sec:concl}

We have considered the one-parameter stochastic process describing the one-dimensional movement (projection on the $x$-axis) of a single particle in two-dimensional MCA. In order to find the distribution of the $x$-coordinate of the particle at every discrete time step (the probability distribution of the random variable $X_t$), we have introduced the additional random variable $\Delta_t$ associated with the direction of the movement of the particle. Using our approach, we have reduced this process to the two-dimensional Markov chain of $X_t$ and $\Delta_t$ and have obtained the desired distribution by the probability generating function method. In particular, we have derived the analytic formula for the probability distribution (Theorem \ref{teor1}) and for two first moments \eqref{prgen13} of $X_t$. The probability distribution is expressed in compact form in terms of Jacobi polynomials.

The primary results of this work are the Theorem {\ref{teor1}} and formula {\eqref{disc1}}. The formula {\eqref{disc1}} gives the diffusion coefficient of the one-parameter two-dimensional MCA. It establishes the correspondence between the cellular automaton model of the diffusion and the kinetic theory. Note that in our formalist diffusion coefficient is given by the asymptotic behaviour of the dispersion of $X_t$ for the distribution under consideration. It is shown that the formula \eqref{disc1} agrees with earlier empirical results of other authors in the specified cases. Thus, the formula \eqref{disc1} is a fundamental mathematical result for modelling the diffusion using MCA. Theorem {\ref{teor1}} gives the probability distribution of a single particle in MCA that is given by the exact solutions to the two-dimensional Markov chain {\eqref{def4}}, {\eqref{def5}}. While {\eqref{disc1}} is of main applied interest, Theorem {\ref{teor1}} is a more general fundamental result. It allows us to study specific features of various realizations of MCA discussed in Section 6. Moreover, the exactly solvable two-dimensional Markov chain are also valuable for the fundamental mathematics. In view of Theorem {\ref{teor1}}, the empirically proven fact that MCA describes the diffusion can be presented as a specific property of Jacobi polynomials (see Fig. {\ref{fig3}}, {\ref{fig4}}).

Due to the potential interest for other applications, it is worth to study other properties of the process under consideration. Note that the formulae derived in this work are valid for the range $p\in (0;1)$ while only $p\in (0;\frac{1}{2}]$ have a physical meaning in MCA. Therefore, the more detailed study of this process for the parameter $p\in (\frac{1}{2};1)$ is planned in a separate future work. Also, it is of interest to generalize our approach for the three-dimensional MCA, especially the formula {\eqref{disc1}}.

\section*{Acknowledgement}
We are thankful to M.L. Gromov and N.A. Shalyapina for stimulating discussions.

\bibliography{lit1}

\appendix
\section{Proof of Theorem \ref{teor1}}
\label{sec:app}
Let us write the probability generating function \eqref{prgen8} as
\begin{equation}
\begin{gathered}
G_t^{+}(z,t)=\frac{1}{2^{t+1}m(z)}\Big[-\lambda\big[\big(q(z)+m(z)\big)^{t-1} -\big(q(z)-m(z)\big)^{t-1}\big]+ \\
+2(1-p+pz)\Big( \big(q(z)+m(z)\big)^t-\big(q(z)-m(z)\big)^t\Big)  \Big],\\
G_t^{-}(z,t)=\frac{1}{2^{t+1}m(z)}\Big[ \big(q(z)+m(z)\big)^{t+1}-\big(q(z)-m(z)\big)^{t+1}+\\
+2(1-p-pz)\Big( \big(q(z)+m(z)\big)^t-\big(q(z)-m(z)\big)^t\Big)\Big],
\end{gathered}
\label{app1}
\end{equation}
where $q(z)=p\left(z+\dac{1}{z}\right)$, $\lambda=4(2p-1)$, and $m(z)$ is defined in \eqref{prgen7}.
Next, we use the following properties:
\begin{equation}
\begin{gathered}
\left(q+m\right)^t +\left(q-m\right)^t=
\left\{\begin{array}{l}2\displaystyle\sum_{k=0}^{n}C^{2k}_{2n} q^{2k} m^{2n-2k}, \qquad t=2n,\cr
2\displaystyle\sum_{k=0}^{n}C^{2k+1}_{2n+1} q^{2k+1} m^{2n-2k}, \qquad t=2n+1,
\end{array}\right.\\
\dac{\left(q+m\right)^t -\left(q-m\right)^t}{m}=
\left\{\begin{array}{l}2\displaystyle\sum_{k=0}^{n-1}C^{2k+1}_{2n} q^{2k+1} m^{2n-2k-2}, \qquad t=2n,\cr
2\displaystyle\sum_{k=0}^{n}C^{2k}_{2n+1} q^{2k} m^{2n-2k}, \qquad t=2n+1,
\end{array}\right.
\end{gathered}
\label{app2}
\end{equation}
and arrive at the following formulae:
\begin{equation}
\begin{gathered}
G_{2n+1}^{+}(z)=\frac{1}{2^{2n+1}}\Big\{-\lambda\sum_{j=0}^{n-1}q(z)^{2(n-j)-1} m(z)^{2j}
C^{2j+1}_{2n}+\\
+2(1-p+pz)\sum_{j=0}^{n}q(z)^{2(n-j)}m(z)^{2j}C^{2j+1}_{2n+1} \Big\},\\
G_{2n}^{+}(z)=\frac{1}{2^{2n}}\Big\{-\lambda\sum_{j=0}^{n-1}q(z)^{2j} m(z)^{2(n-1-j)}
C^{2j}_{2n-1}+\\
+2(1-p+pz)\sum_{j=0}^{n-1}q(z)^{2j+1}m(z)^{2(n-j)-2}C^{2j+1}_{2n} \Big\},\\
G_{2n+1}^{-}(z)=\frac{1}{2^{2n+1}}\Big\{\sum_{j=0}^{n}q(z)^{2(n-j)+1} m(z)^{2j}
C^{2j+1}_{2n+2}+\\
+2(1-p-pz)\sum_{j=0}^{n}q(z)^{2(n-j)}m(z)^{2j}C^{2j+1}_{2n+1} \Big\},\\
G_{2n}^{-}(z)=\frac{1}{2^{2n}}\Big\{\sum_{j=0}^{n}q(z)^{2j} m(z)^{2(n-j)}
C^{2j}_{2n+1}+\\
+2(1-p-pz)\sum_{j=0}^{n-1}q(z)^{2j+1}m(z)^{2(n-j)-2}C^{2j+1}_{2n} \Big\}.
\end{gathered}
\label{app3}
\end{equation}
Using the relation $m^2(z)=q^2(z)-\lambda$, the binomial theorem, and changing the order of summation, we transform \eqref{app3} to the power series in $q(z)$:
\begin{equation}
\begin{gathered}
G_{2n+1}^{+}(z)=\frac{1}{2^{2n+1}}\Big\{\sum_{k=0}^{n-1}
q(z)^{2(n-k)-1}(-1)^{k+1}\lambda^{k+1}R(n-1,k)+ \\
+2(1-p+pz)\sum_{k=0}^{n}q(z)^{2(n-k)}(-1)^k\lambda^k Q(n,k) \Big\},\\
G_{2n}^{+}(z)=\frac{1}{2^{2n}}\Big\{\sum_{k=0}^{n-1}
q(z)^{2(n-k-1)}(-1)^{k+1}\lambda^{k+1}Q(n-1,k)+\\
+2(1-p+pz)\sum_{k=0}^{n-1}q(z)^{2(n-k)-1}(-1)^k\lambda^k R(n-1,k)\Big\},\\
G_{2n+1}^{-}(z)=\frac{1}{2^{2n+1}}\Big\{\sum_{k=0}^{n}
q(z)^{2(n-k)+1}(-1)^k\lambda^kR(n,k)+\\
+2(1-p-pz)\sum_{k=0}^{n}q(z)^{2(n-k)}(-1)^k\lambda^k Q(n,k) \Big\},\\
G_{2n}^{-}(z)=\frac{1}{2^{2n}}\Big\{\sum_{k=0}^{n}
q(z)^{2(n-k)}(-1)^k\lambda^k Q(n,k)+ \\
+2(1-p-pz)\sum_{k=0}^{n-1}q(z)^{2(n-k)-1}(-1)^k\lambda^k R(n-1,k) \Big\},
\end{gathered}
\label{app4}
\end{equation}
where the numbers $Q$, $R$ are given by
\begin{equation}
\begin{gathered}
R(n,k)=\sum_{j=0}^{n-k}C^{2j+1}_{2n+2}C^k_{n-j}, \qquad Q(n,k)=\sum_{j=0}^{n-k}C^{2j}_{2n+1}C^k_{n-j}, \\
R(n,-1)=Q(n,-1)\equiv 0.
\end{gathered}
\label{app5}
\end{equation}

Using the binomial theorem for the power of $q(z)$, changing the order of summation, in view of \eqref{prgen3}, we obtain the following identities:
\begin{equation}
\begin{gathered}
G_{2n+1}(z)=\frac{1}{2^{2n+2}}\Big\{\sum_{j=0}^{n} \Big(\sum_{l=j}^{n}p^{2l+1} a_{n,l} C^{l-j}_{2l+1} \Big) z^{2j+1}+ \\
+\sum_{j=-n-1}^{-1} \Big(\sum_{l=-j-1}^{n}p^{2l+1} a_{n,l} C^{l-j}_{2l+1} \Big) z^{2j+1}+ \\
+ 4(1-p)\sum_{j=-n}^{n}\Big(\sum_{l=|j|}^{n}p^{2l} b_{n,l}C_{2l}^{l-j}\Big) z^{2j}\Big\}, \\
G_{2n}(z)=\frac{1}{2^{2n+1}}\Big\{\sum_{j=-n}^{n}\Big(\sum_{l=|j|}^{n}p^{2l} r_{n,l}C^{l-j}_{2l}\Big) z^{2j}+ \\
+ 4(1-p)\sum_{j=0}^{n-1}\Big(\sum_{l=j}^{n-1}p^{2l+1}d_{n-1,l} C^{l-j}_{2l+1}\Big) z^{2j+1}+ \\
+ 4(1-p)\sum_{j=-n}^{-1}\Big(\sum_{l=-j-1}^{n-1}p^{2l+1}d_{n-1,l} C^{l-j}_{2l+1}\Big) z^{2j+1} \Big\},
\end{gathered}
\label{dolem1}
\end{equation}
where
\begin{equation}
\begin{gathered}
a_{n,l}=(-1)^{n-l}\lambda^{n-l}\big[ R(n-1,n-1-l)+R(n,n-l)\big],\\
b_{n,l}=(-1)^{n-l}\lambda^{n-l} Q(n,n-l),\\
r_{n,l}=(-1)^{n-l}\lambda^{n-l}\big[Q(n-1,n-1-l)+Q(n,n-l)\big],\\
d_{n,l}=(-1)^{n-l}\lambda^{n-l} R(n,n-l).
\end{gathered}
\label{dolem2}
\end{equation}

In view of \eqref{prgen1}, we have
\begin{equation}
\begin{gathered}
P_{2n}(2j)=\frac{1}{2^{2n+1}} \sum_{l=j}^{n}p^{2l} r_{n,l}C^{l-j}_{2l}, \qquad j=\overline{0,n};\\
P_{2n}(2j+1)=\frac{1-p}{2^{2n-1}}\sum_{l=j}^{n-1}p^{2l+1}d_{n-1,l} C^{l-j}_{2l+1}, \qquad j=\overline{0,n-1};\\
P_{2n+1}(2j)=\frac{1-p}{2^{2n}}\sum_{l=j}^{n}p^{2l} b_{n,l}C_{2l}^{l-j}, \qquad j=\overline{0,n};\\
P_{2n+1}(2j+1)=\frac{1}{2^{2n+2}}\sum_{l=j}^{n}p^{2l+1} a_{n,l} C^{l-j}_{2l+1}, \qquad j=\overline{0,n}.
\end{gathered}
\label{dolem3}
\end{equation}

In order to simplify \eqref{dolem3}, let us prove the lemma below.

\begin{lemma}
The following relation holds true:
\begin{equation}
\sum_{j=0}^{n-k}C^{2j}_{2n+1}C^k_{n-j}=2^{2(n-k)}C^{k}_{2n-k}.
\label{app5a}
\end{equation}
\label{lemma1}
\end{lemma}

\begin{proof}
Let us define
\begin{equation}
T(n,k,j)=C^{2j+1}_{2n+2}C^k_{n-j}=\dac{(2n+1)!}{(2j)!(2n-2j+1)!}\dac{(n-j)!}{k!(n-j-k)!}.
\label{app6}
\end{equation}
Using the relation \cite{Bateman1}
\begin{equation}
(2n)!=1^{\overline{2n}}=4^n \left(\dac{1}{2}\right)^{\overline{n}} 1^{\overline{j}},
\label{app7}
\end{equation}
where $x^{\overline{n}}$ is the Pochhammer function (upper factorial), one readily gets
\begin{equation}
T(n,k,j)=(-1)^j\dac{\left(-n-\dac{1}{2}\right)^{\overline{j}}n!}{\left(\dac{1}{2}\right)^{\overline{j}}j!k!(n-j-k)!}.
\label{app8}
\end{equation}

Then, we have
\begin{equation}
Q(n,k)=\sum_{j=0}^{n-k}T(n,k,j)=C^k_n\cdot {}_2F_1\left[k-n,-n-\dac{1}{2},\dac{1}{2},1\right].
\label{app9}
\end{equation}
Here, we have used the formula
\begin{equation}
\sum_{j=0}^{n} C_{n}^k (-1)^j \dac{a^{\overline{j}}}{b^{\overline{j}}} z^n={}_2F_1\left[-n,a,b,z\right],
\label{app9a}
\end{equation}
where ${}_2F_1\left[m,a,b,z\right]$ is the hypergeometric function \cite{srivastava21,Srivastava2022}.

The Gauss summation theorem yields
\begin{equation}
{}_2F_1\left[k-n,-n-\dac{1}{2},\dac{1}{2},1\right]=\dac{\Gamma\left[\dac{1}{2}\right]\Gamma\left[1+2n-k\right]}{\Gamma\left[\dac{1}{2}+n-k\right]\Gamma\left[1+n\right]}=2^{2(n-k)}\dac{(2n-k)!(n-k)!}{(2n-2k)!n!},
\label{app10}
\end{equation}
where $\Gamma[x]$ is the gamma function. Then, \eqref{app9} and \eqref{app10} yield \eqref{app5a}.
\end{proof}

\begin{corollary}
The numbers $Q(n,k)$ and $R(n,k)$ given by \eqref{app5} can be generalized to the functions of a half-integer $n$ that are related as follows:
\begin{equation}
R(n,k)=Q(n+\frac{1}{2},k).
\label{coreq1}
\end{equation}
\label{corol1}
\end{corollary}

\begin{proof}
From \eqref{app5} and Lemma \ref{lemma1}, we have
\begin{equation}
Q(n,k)=2^{2(n-k)}C^k_{2n-k}.
\label{dop1}
\end{equation}
In a similar way, one can derive the formula
\begin{equation}
R(n,k)=2^{2(n-k)+1}C^k_{2n-k+1},
\label{dop2}
\end{equation}
Assuming $n$ to be a half-integer number, one readily gets \eqref{coreq1} from \eqref{dop1} and \eqref{dop2}.
\end{proof}

Let us consider $P_{2n+1}(2j)$ from \eqref{dolem3}. In view of \eqref{app5a}, \eqref{coreq1}, and \eqref{dolem2}, substituting $l=n-k$, we obtain
\begin{equation}
P_{2n+1}(2j)=(1-p)p^{2j}(1-2p)^{j-n}\sum^{n-j}_{k=0}(-1)^{k}\left(\dac{p^2}{2p-1}\right)^k\frac{(n+k+j)!}{k!(2j+k)!(n-j-k)!}.
\label{app11}
\end{equation}
Using \eqref{app9a}, we arrive at
\begin{equation}
\begin{array}{l}
P_{2n+1}(2j)=(1-p)p^{2j}(1-2p)^{n-j}C_{j+n}^{2j}\cdot  {}_2F_1\left[j-n,1+j+n,1+2j,\frac{p^2}{2p-1}\right], \cr j=\overline{0,n}.
\end{array}
\label{app12}
\end{equation}
In a similar way, one readily gets
\begin{equation}
\begin{array}{l}
P_{2n}(2j)=\dac{1}{2}p^{2j}(1-2p)^{n-j}\Bigg(C_{j+n}^{2j}\cdot  {}_2F_1\left[j-n,1+j+n,1+2j,\frac{p^2}{2p-1}\right]+\cr
+C_{n+j-1}^{2j}\cdot  {}_2F_1\left[j-n+1,j+n,1+2j,\frac{p^2}{2p-1}\right]\Bigg), \qquad j=\overline{0,n-1};\cr \cr
P_{2n}(2j+1)=(1-p)p^{2j+1}(1-2p)^{n-j-1}C_{j+n}^{2j+1}\times\\
\times  {}_2F_1\left[j-n+1,1+j+n,2+2j,\frac{p^2}{2p-1}\right], \qquad j=\overline{0,n-1};\cr \cr
P_{2n+1}(2j+1)=\dac{1}{2}p^{2j+1}(1-2p)^{n-j}\times \cr
\times\Bigg(C_{j+n}^{2j+1}\cdot {}_2F_1\left[j-n+1,1+j+n,2+2j,\frac{p^2}{2p-1}\right]+\cr
+C_{j+n+1}^{2j+1}\cdot {}_2F_1\left[j-n,2+j+n,2+2j,\frac{p^2}{2p-1}\right]\Bigg), \qquad j=\overline{0,n-1}.
\end{array}
\label{app13}
\end{equation}
Finally, using the identities \cite{Bateman2}
\begin{equation}
\begin{gathered}
P_n^{(\alpha,\beta)}(z)=\dac{(\alpha+1)^{\overline{n}}}{n!}\cdot {}_2F_1\left[-n,1+\alpha+\beta+n,\alpha+1,\frac{1}{2}(1-z)\right],\\
\dac{n+\frac{1}{2}\alpha+1}{n+1}(1+x)P_n^{(\alpha,1)}(x)=P_{n+1}^{(\alpha,0)}(x)+P_{n}^{(\alpha,0)}(x),
\end{gathered}
\label{app14}
\end{equation}
we arrive at the formulae \eqref{teo1}, \eqref{teo2}, \eqref{teo3}, \eqref{teo4}.

The probability $P_n(n)$ \eqref{teo5} can be readily obtained if one notice that the special case $X_t=t$ for the given $t>0$ is equivalent to $\Delta_t=+1$ $\forall t\in {\mathbb{N}}$.

For $p=\dac{1}{2}$, \eqref{prgen8} and \eqref{prgen3} yield
\begin{equation}
G_t(z)=2^{-t}\left(\dac{1}{z}+2+z\right)\sum_{k=0}^{t-1}C_{t-1}^k s^{2k-t+1}, \qquad t\in {\mathbb{N}}.
\label{app15}
\end{equation}
From \eqref{app15} and \eqref{prgen1}, we have \eqref{teo6}.

\clearpage

\begin{figure}[h]
\includegraphics[width=10.5 cm]{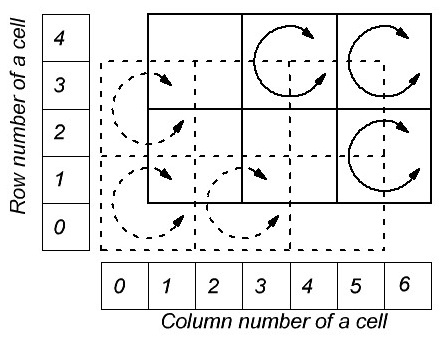}
\caption{Odd and even partitions of the cell grid. Solid lines divide the cell grid into blocks of the odd partition and dashed lines divide the cell grid into blocks of the even partition. Arrows demonstrate the rotation {\it rule}. \label{fig1}}
\end{figure}

\begin{figure}[h]
\includegraphics[width=10.5 cm]{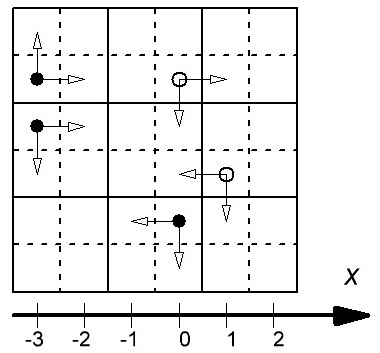}
\caption{Examples of possible movements of a particle in the MCA. The black dot is for the particle movement at an odd time step (the rotation {\it rule} applies for blocks of the odd partition) and the white dot is for the particle movement at an even time step (the rotation {\it rule} applies for block of the even partition). \label{fig2}}
\end{figure}

\begin{figure}[h]
\begin{minipage}[b][][b]{0.9\linewidth}\centering
    \includegraphics[width=12.5 cm]{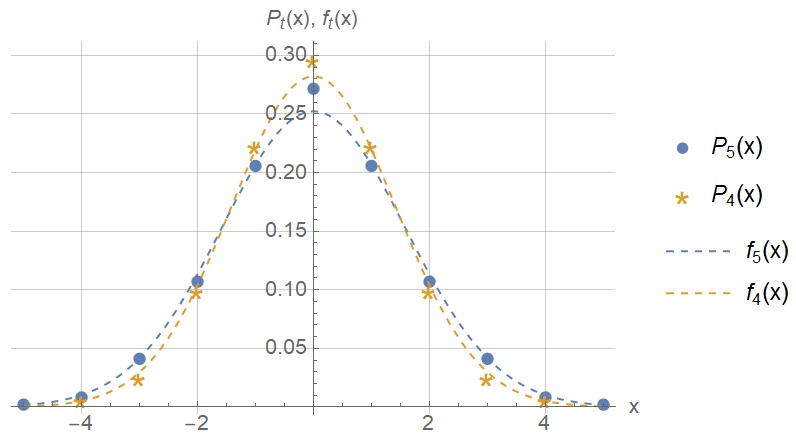} \\ a)
  \end{minipage}\\
  \begin{minipage}[b][][b]{0.9\linewidth}\centering
    \includegraphics[width=12.5 cm]{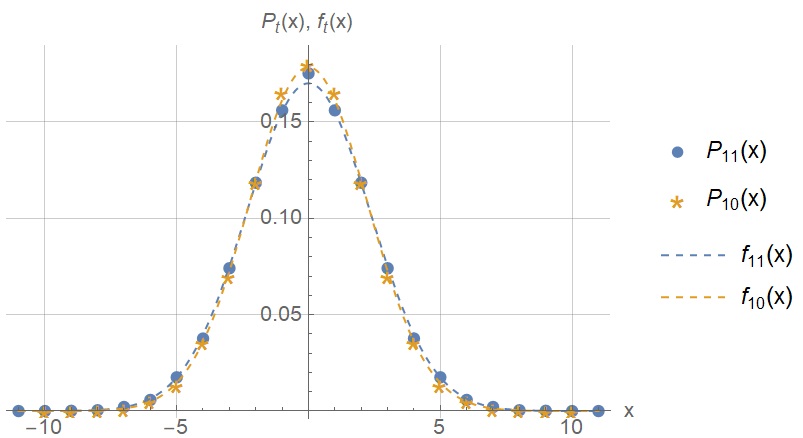} \\ b)
  \end{minipage}
\caption{The plot of the probability distribution $P_{t}(x)$ along with the plot of the normal probability distribution function $f_t(x)$ for $p=\dac{1}{3}$.\label{fig3}}
\end{figure}

\begin{figure}[h]
\begin{minipage}[b][][b]{0.9\linewidth}\centering
    \includegraphics[width=12.5 cm]{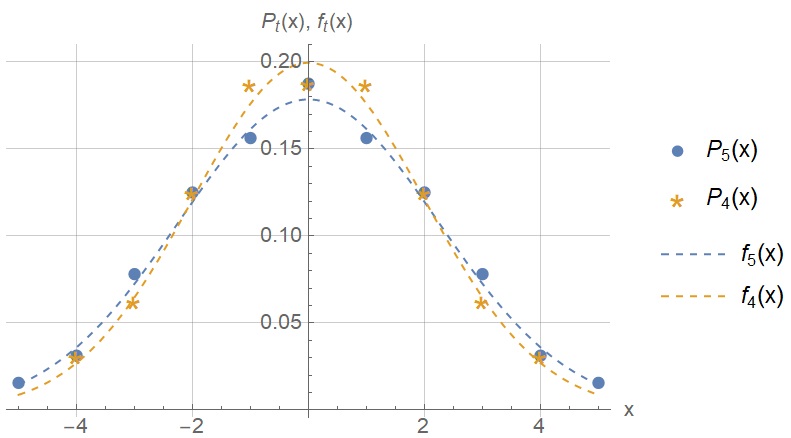} \\ a)
  \end{minipage}\\
  \begin{minipage}[b][][b]{0.9\linewidth}\centering
    \includegraphics[width=12.5 cm]{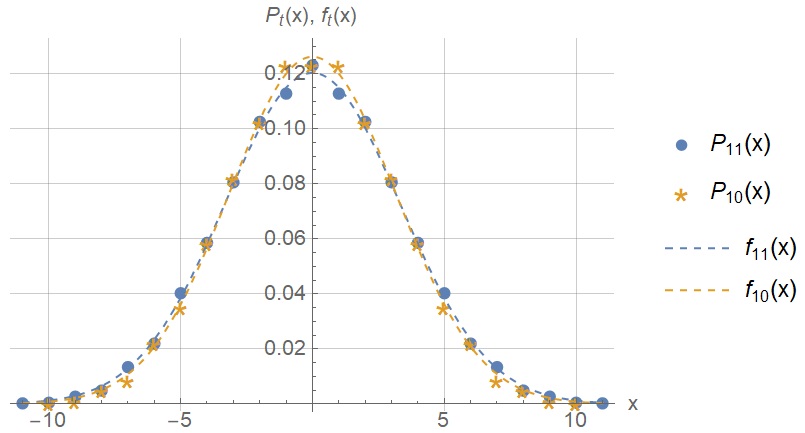} \\ b)
  \end{minipage}
\caption{The plot of the probability distribution $P_{t}(x)$ along with the plot of the normal probability distribution function $f_t(x)$ for $p=\dac{1}{2}$.\label{fig4}}
\end{figure}

\begin{figure}[h]
\begin{minipage}[b][][b]{0.9\linewidth}\centering
    \includegraphics[width=12.5 cm]{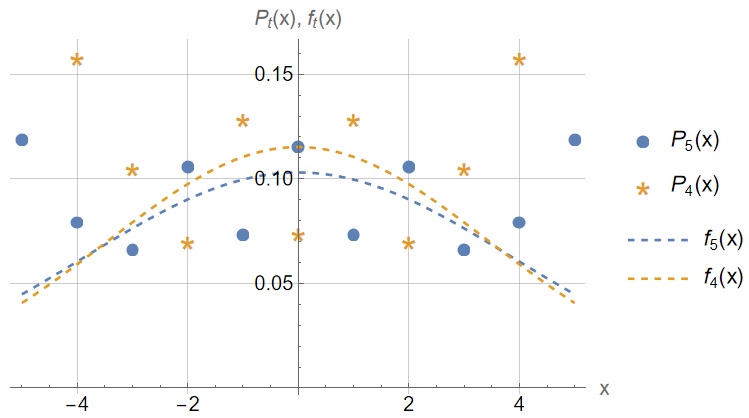} \\ a)
  \end{minipage}\\
  \begin{minipage}[b][][b]{0.9\linewidth}\centering
    \includegraphics[width=12.5 cm]{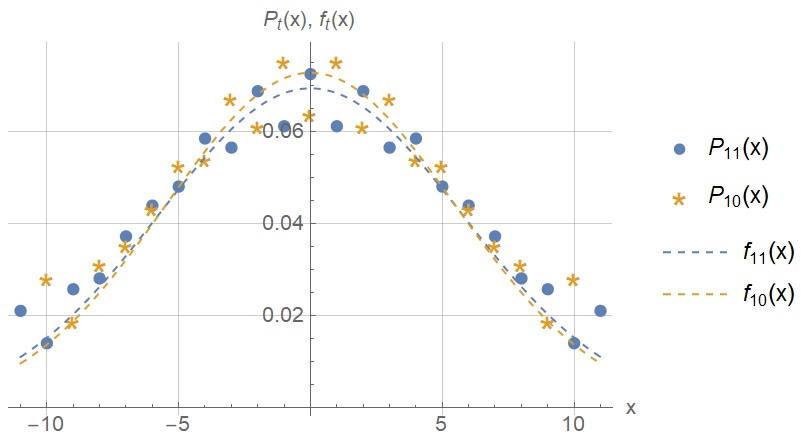} \\ b)
  \end{minipage}
\caption{The plot of the probability distribution $P_{t}(x)$ along with the plot of the normal probability distribution function $f_t(x)$ for $p=\dac{3}{4}$.\label{fig5}}
\end{figure}

\begin{figure}[h]
\begin{minipage}[b][][b]{1.0\linewidth}\centering
    \includegraphics[width=12.5 cm]{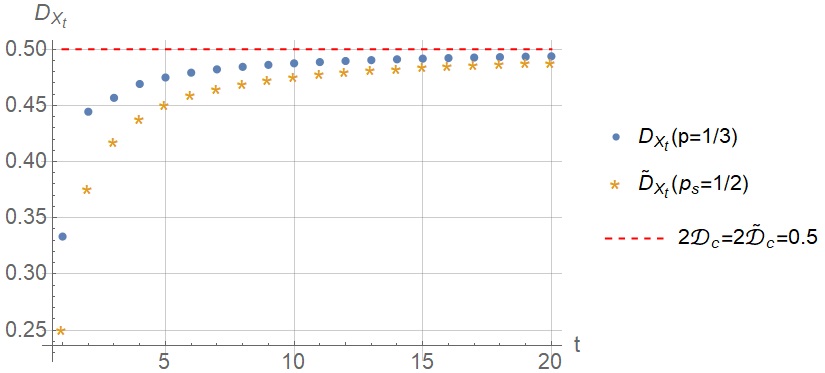} \\ a)
  \end{minipage}\\
  \begin{minipage}[b][][b]{1.0\linewidth}\centering
    \includegraphics[width=12.5 cm]{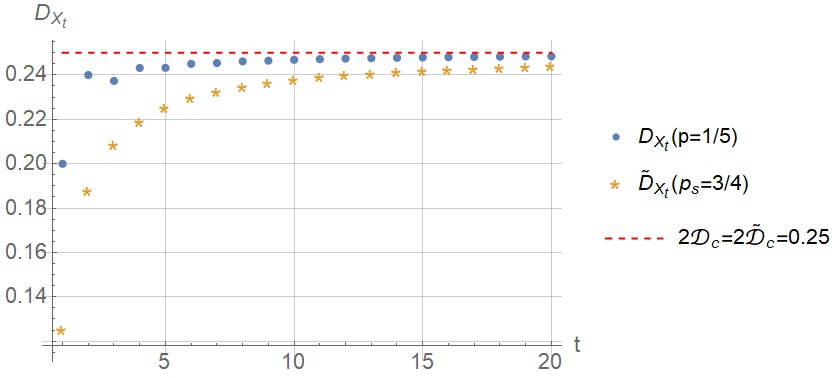} \\ b)
  \end{minipage}
\caption{The time plot of the dispersion of the $X_t$ for two types of the MCA. It illustrates the macroscopic behaviour of these MCA at small times. The parameters in (a) correspond to MCA with the diffusion coefficient of $0.25$, and the parameters in (b) correspond to MCA with the diffusion coefficient of $0.125$. \label{fig6}}
\end{figure}

\end{document}